\documentclass[11pt]{amsart}
\usepackage{amsmath, amssymb}
\usepackage{amscd}
\usepackage{verbatim}
\usepackage{amssymb,latexsym,amsmath,amsfonts}
\usepackage{latexsym}
\usepackage[mathscr]{eucal}
\usepackage{bm}

\title[Decomposition of a $\Gamma_3$-contraction]
{On decomposition of operators having $\Gamma_3$ as a spectral
set}

\author{Sourav Pal}
\address{Department of Mathematics, Indian Institute of Technology Bombay, Mumbai - 400076, India.}
\email{sourav@math.iitb.ac.in,
souravmaths@gmail.com}
\thanks{The author is supported by Seed Grant of IIT Bombay, CPDA and the INSPIRE Faculty Award
(Award No. DST/INSPIRE/04/2014/001462) of DST, India.}



\def ~{\hspace{1mm}}

\def\textmatrix#1&#2\\#3&#4\\{\bigl({#1 \atop #3}\ {#2 \atop #4}\bigr)}
\def\dispmatrix#1&#2\\#3&#4\\{\left({#1 \atop #3}\ {#2 \atop #4}\right)}
\newcommand{\beg}{\begin{equation}}
\newcommand{\eeg}{\end{equation}}
\newcommand{\ben}{\begin{eqnarray*}}
\newcommand{\een}{\end{eqnarray*}}

\newtheorem{thm}{Theorem}[section]

\newtheorem{prop}[thm]{Proposition}
\numberwithin{equation}{section}
\theoremstyle{definition}
\newtheorem{defn}[thm]{Definition}

\def\textmatrix#1&#2\\#3&#4\\{\bigl({#1 \atop #3}\ {#2 \atop #4}\bigr)}
\def\dispmatrix#1&#2\\#3&#4\\{\left({#1 \atop #3}\ {#2 \atop #4}\right)}

\begin{document}

\begin{abstract}
The symmetrized polydisc of dimension three is the set
\[
\Gamma_3 =\{ (z_1+z_2+z_3, z_1z_2+z_2z_3+z_3z_1, z_1z_2z_3)\,:\,
|z_i|\leq 1 \,,\, i=1,2,3 \} \subseteq \mathbb C^3\,.
\]
A triple of commuting operators for which $\Gamma_3$ is a spectral
set is called a $\Gamma_3$-contraction. We show that every
$\Gamma_3$-contraction admits a decomposition into a
$\Gamma_3$-unitary and a completely non-unitary
$\Gamma_3$-contraction. This decomposition parallels the canonical
decomposition of a contraction into a unitary and a completely
non-unitary contraction. We also find new characterizations for
the set $\Gamma_3$ and $\Gamma_3$-contractions.
\end{abstract}

\maketitle

\section{Introduction}

One of the most wonderful discoveries in one variable operator
theory is the canonical decomposition of a contraction which
ascertains that every contraction operator (i.e, an operator with
norm not greater than $1$) admits a unique decomposition into two
orthogonal parts of which one is a unitary and the other is a
completely non-unitary contraction. More precisely, for an
operator $T$ with norm not greater than one acting on a Hilbert
space $\mathcal H$, there exist unique reducing subspaces
$\mathcal H_1, \mathcal H_2$ of $T$ such that $\mathcal H=\mathcal
H_1 \oplus \mathcal H_2$, $T|_{\mathcal H_1}$ is a unitary and
$T|_{\mathcal H_2}$ is a completely non-unitary contraction (see
Theorem 3.2 in Ch-I, \cite{nagy} for details). A contraction on a
Hilbert space is said to be \textit{completely non-unitary} if
there is no reducing subspace on which the operator acts like a
unitary. Following von Neumann's famous notion of spectral set for
an operator (which we define below), a contraction is better
understood as an operator having the closed unit disk
$\overline{\mathbb D}$ of the complex plane as a spectral set.
Indeed, in 1951 von Neumann proved the following theorem whose
impact has been extraordinary.

\begin{thm}[von Neumann, \cite{vN}]
An operator $T$ acting on a Hilbert space is a contraction if and
only if the closed unit disk $\overline{\mathbb D}$ is a spectral
set for $T$.
\end{thm}

Since an operator having $\overline{\mathbb D}$ as a spectral set
admits a canonical decomposition, it is naturally asked whether we
can decompose operators having a particular domain in $\mathbb
C^n$ as a spectral set. In \cite{ay-jot}, Agler and Young answered
this question by showing an explicit decomposition of a pair of
commuting operators having the closed symmetrized bidisc
\[
\Gamma_2 =\{(z_1+z_2,z_1z_2)\,:\,|z_i|\leq 1,\, i=1,2\}
\]
 as a spectral set (Theorem 2.8, \cite{ay-jot}). In this article, we provide an analogous
 decomposition for operators having the closed symmetrized tridisc
 \[
\Gamma_3 =\{ (z_1+z_2+z_3, z_1z_2+z_2z_3+z_3z_1, z_1z_2z_3)\,:\,
|z_i|\leq 1 \,,\, i=1,2,3 \}
 \]
as a spectral set. The reason behind considering the symmetrized
polydisc of dimension $3$ in particular is that there are
substantial variations in operator theory if we move from two to
three dimensional symmetrized polydisc, e.g., rational dilation
succeeds on the symmetrized bidisc \cite{ay-jfa, tirtha-sourav,
sourav} but fails on the symmetrized tridisc, \cite{sourav1}.
This article can be considered as a sequel of \cite{sourav1}.\\

A compact subset $X$ of $\mathbb C^n$ is said to be a
\textit{spectral set} for a commuting $n$-tuple of bounded
operators $\underline{T}=(T_1,\hdots,T_n)$ defined on a Hilbert
space $\mathcal H$ if the Taylor joint spectrum
$\sigma_T(\underline{T})$ of $\underline{T}$ is a subset of $X$
and
\[
\|f(\underline{T})\|\leq \|f\|_{\infty,
X}=\sup\{|f(z_1,\hdots,z_n)|\,:\,(z_1,\hdots,z_n)\in\ X\}\,,
\]
for all rational functions $f$ in $\mathcal R(X)$. Here $\mathcal
R(X)$ denotes the algebra of all rational functions on $X$, that
is, all quotients $p/q$ of holomorphic polynomials $p,q$ in
$n$-variables for which $q$ has no zeros in $X$.\\

For $n\geq 2$, the symmetrization map in $n$-complex variables
$z=(z_1,\dots,z_n)$ is the following proper holomorphic map
\[
\pi_n(z)=(s_1(z),\dots, s_{n-1}(z), p(z))
\]
 where
 \[
s_i(z)= \sum_{1\leq k_1 \leq k_2 \dots \leq k_i \leq n-1}
z_{k_1}\dots z_{k_i} \quad \text{ and } p(z)=\prod_{i=1}^{n}z_i\,.
 \]
The closed \textit{symmetrized} $n$-\textit{disk} (or simply
closed \textit{symmetrized polydisc}) is the image of the closed
unit $n$-disc $\overline{\mathbb D^n}$ under the symmetrization
map $\pi_n$, that is, $\Gamma_n := \pi_n(\overline{\mathbb D^n})$.
Similarly the open symmetrized polydisc $\mathbb G_n$ is defined
as the image of the open unit polydisc $\mathbb D^n$ under
$\pi_n$. The set $\Gamma_n$ is polynomially convex but not convex
(see \cite{edi-zwo, BSR}). So in particular the closed and open
symmetrized tridisc are the sets
\begin{align*}
\Gamma_3 &=\{ (z_1+z_2+z_3,z_1z_2+z_2z_3+z_3z_1,z_1z_2z_3):
\,|z_i|\leq 1, i=1,2,3 \} \subseteq \mathbb C^3 \\ \mathbb G_3 &
=\{ (z_1+z_2+z_3,z_1z_2+z_2z_3+z_3z_1,z_1z_2z_3): \,|z_i|< 1,
i=1,2,3 \}\subseteq \Gamma_3.
\end{align*}

We obtain from the literature (see \cite{edi-zwo, BSR}) the fact
that the distinguished boundary of the symmetrized polydisc is the
symmetrization of the distinguished boundary of the
$n$-dimensional polydisc, which is $n$-torus $\mathbb T^n$. Hence
the distinguished boundary for $\Gamma_3$ is the set
\begin{align*}
b\Gamma_3 &=\{(z_1+z_2+z_3,z_1z_2+z_2z_3+z_3z_1,z_1z_2z_3):
\,|z_i|=1, i=1,2,3\}.
\end{align*}

Operator theory on the symmetrized polydiscs of dimension $2$ and
$n$ have been extensively studied in past two decades
\cite{ay-jfa, ay-jot, AY, tirtha-sourav, tirtha-sourav1, BSR,
sourav, pal-shalit}.

\begin{defn}
A triple of commuting operators $(S_1,S_2,P)$ on a Hilbert space
$\mathcal H$ for which $\Gamma_3$ is a spectral set is called a
$\Gamma_3$-$contraction$. A $\Gamma_3$-contraction $(S_1,S_2,P)$
is said to a \textit{completely non-unitary} if $P$ is a
completely non-unitary contraction.
\end{defn}

It is evident from the definition that if $(S_1,S_2,P)$ is a
$\Gamma_3$-contraction then $S_1,S_2$ have norms not greater than
$3$ and $P$ is a contraction. Unitaries, isometries and
co-isometries are important special classes of contractions. There
are natural analogues of these classes for
$\Gamma_3$-contractions.

\begin{defn}
Let $S_1,S_2,P$ be commuting operators on a Hilbert space
$\mathcal H$. We say that $(S_1,S_2,P)$ is
\begin{itemize}
\item [(i)] a $\Gamma_3$-\textit{unitary} if $S_1,S_2,P$ are
normal operators and the Taylor joint spectrum
$\sigma_T(S_1,S_2,P)$ is contained in $b\Gamma_3$ ; \item [(ii)] a
$\Gamma_3$-\textit{isometry} if there exists a Hilbert space
$\mathcal K$ containing $\mathcal H$ and a $\Gamma_3$-unitary
$(\tilde{S_1},\tilde{S_2},\tilde{P})$ on $\mathcal K$ such that
$\mathcal H$ is a common invariant subspace for
$\tilde{S_1},\tilde{S_2},\tilde{P}$ and that
$S_i=\tilde{S_i}|_{\mathcal H}$ for $i=1,2$ and
$\tilde{P}|_{\mathcal H}=P$; \item [(iii)] a
$\Gamma_3$-\textit{co-isometry} if $(S_1^*,S_2^*,P^*)$ is a
$\Gamma_3$-isometry.
\end{itemize}
\end{defn}

Moreover, a $\Gamma_3$-isometry $(S_1,S_2,P)$ is said to be
\textit{pure} if $P$ is a pure contraction, that is,
$P^*\rightarrow 0$ strongly
as $n\rightarrow \infty$.\\

The main result of this article is the following explicit
orthogonal decomposition of a $\Gamma_3$-contraction which
parallels the one-variable canonical decomposition.

\begin{thm}\label{thm:decomp}
Let $(S_1,S_2,P)$ be a $\Gamma_3$-contraction on a Hilbert space
$\mathcal H$. Let $\mathcal H_1$ be the maximal subspace of
$\mathcal H$ which reduces $P$ and on which $P$ is unitary. Let
$\mathcal H_2=\mathcal H\ominus \mathcal H_1$. Then $\mathcal
H_1,\mathcal H_2$ reduce $S_1, S_2$; $(S_1|_{\mathcal
H_1},S_2|_{\mathcal H_1},P|_{\mathcal H_1})$ is a
$\Gamma_3$-unitary and $(S_1|_{\mathcal H_2},S_2|_{\mathcal
H_2},P|_{\mathcal H_2})$ is a completely non-unitary
$\Gamma_3$-contraction. The subspaces $\mathcal H_1$ or $\mathcal
H_2$ may equal to the trivial subspace $\{0\}$.
\end{thm}

En route we find few characterizations for the set $\Gamma_3$ and
also for the $\Gamma_3$-contractions which we accumulate in
section 2.

\section{Background material}

In this section we recall some results from literature about the
geometry and operator theory on the set $\Gamma_3$. Also we obtain
few new results in the same direction which we accumulate here. We
begin with a few characterizations of the set $\Gamma_3$.

\begin{thm}\label{thm:sc1}

Let $(s_1,s_2,p)\in \mathbb C^3$. Then the following are
equivalent:
\begin{enumerate}
\item $(s_1,s_2,p)\in\Gamma_3$\,; \item $(\omega s_1,\omega^2 s_2,
\omega^3 p)\in \Gamma_3$ for all $\omega\in\mathbb T$ \,;  \item
$|p|\leq 1$ and there exists $(c_1,c_2)\in \Gamma_2$ such that

\[
s_1=c_1+\bar{c_2}p \text{ and } s_2=c_2+\bar{c_1}p,
\]
where $\Gamma_2$ is the closed symmetrized bidisc defined as
\[
\Gamma_2 =\{ (z_1+z_2,z_1z_2)\,:\, z_1,z_2\in \overline{\mathbb D}
\}.
\]
\end{enumerate}

\end{thm}

\begin{proof}
(1)$\Leftrightarrow (3)$ has been established in \cite{costara1}
(see Theorem 3.7 in \cite{costara1} for a proof). We prove here
(1)$\Leftrightarrow (2)$.  Let $(s_1,s_2,p)\in \Gamma_3$. Then by
(1)$\Leftrightarrow (3)$, $|p|\leq 1$ and there exist
$(c_1,c_2)\in \Gamma_2$ such that
\[
s_1=c_1+\bar{c_2}p,\quad s_2=c_2+\bar{c_1}p\,.
\]
Since $(c_1,c_2)\in \Gamma_2$, there are complex numbers $u_1,u_2$
of modulus not greater than $1$ such that $c_1=u_1+u_2$ and
$c_2=u_1u_2$. For $\omega\in\mathbb T$ if we choose $d_1=\omega
c_1 \text{ and } d_2=\omega^2 c_2$ we see that
\[
d_1=\omega u_1+\omega u_2 \text{ and } d_2 = (\omega u_1)(\omega
u_2)\,,
\]
which means that $(d_1,d_2) \in \Gamma_2$. Now
\begin{align*}
& \omega s_1 =\omega(c_1+\bar{c_2}p)=\omega c_1+\overline{\omega^2
c_2}(\omega^3 p)=d_1+\bar{d_2}(\omega^3 p)\,, \\& \omega^2
s_2=\omega^2(c_2+\bar{c_1}p)=\omega^2 c_2+\overline{\omega
c_1}(\omega^3 p)=d_2+\bar{d_1}(\omega^3 p).
\end{align*}
Therefore, by part (1)$\Leftrightarrow (3)$, $(\omega s_1,\omega^2
s_2, \omega^3 p)\in \Gamma_3$. The other side of the proof is
trivial.

\end{proof}

In a similar fashion, we have the following characterizations for
$\Gamma_3$-contractions.

\begin{thm}\label{thm:sc2}
Let $(S_1,S_2,P)$ be a triple of commuting operators acting on a
Hilbert space $\mathcal H$. Then the following are equivalent:

\begin{enumerate}
\item $(S_1,S_2,P)$ is a $\Gamma_3$-contraction\,; \item for all
holomorphic polynomials $f$ in three variables

\[
\|f(S_1,S_2,P)\|\leq \|f\|_{\infty,\Gamma_3}=\sup
\{|f(s_1,s_2,p)|\,:\,(s_1,s_2,p)\in\Gamma_3\} \,;
\]

\item $(\omega S_1,\omega^2 S_2,\omega^3 P)$ is a
$\Gamma_3$-contraction for any $\omega\in\mathbb T$.
\end{enumerate}

\end{thm}

\begin{proof}

$(1)\Rightarrow (2)$ follows from definition of spectral set and
$(2)\Rightarrow (1)$ just requires polynomial convexity of the set
$\Gamma_3$. We prove here $(1)\Rightarrow (3)$ because
$(3)\Rightarrow (1)$ is obvious. Let $f(s_1,s_2,p)$ be a
holomorphic polynomial in the co-ordinates of $\Gamma_3$ and for
$\omega\in\mathbb T$ let $ f_1(s_1,s_2,p)=f(\omega s_1,\omega^2
s_2,\omega^3 p)$. It is evident from part $(1)\Rightarrow (2)$
that
\[
\sup\{|f(s_1,s_2,p)|\,:\,(s_1,s_2,p)\in\Gamma_3
\}=\sup\{|f_1(s_1,s_2,p)|\,:\,(s_1,s_2,p)\in \Gamma_3\}.
\]
Therefore,
\begin{align*}
\|f(\omega S_1, \omega^2 S_2, \omega^3 P)\|& =\|f_1(S_1,S_2,P)\|
\\& \leq \|f_1\|_{\infty, \Gamma_3} \\& =\|f\|_{\infty,\Gamma_3}.
\end{align*}
Therefore, by $(1)\Rightarrow (2)$, $(\omega S_1, \omega^2 S_2,
\omega^3 P)$ is a $\Gamma_3$-contraction.
\end{proof}

In \cite{sourav1}, two operator pencils $\Phi_1,\, \Phi_2$ were
introduced which played pivotal role in determining the classes of
$\Gamma_3$-contractions for which rational dilation failed or
succeeded. Here we recall the definition of $\Phi_1,\,\Phi_2$ for
any three commuting operators $S_1,S_2,P$ with $\|S_i\|\leq 3$ and
$P$ being a contraction.

\begin{align*}
\Phi_1(S_1,S_2,P)&=9(I-P^*P)+(S_1^*S_1-S_2^*S_2)-6\text{ Re
}(S_1-S_2^*P)\,,\\
\Phi_2(S_1,S_2,P)&=9(I-P^*P)+(S_2^*S_2-S_1^*S_1)-6\text{ Re
}(S_2-S_1^*P)\,.
\end{align*}

The following result whose proof could be found in \cite{sourav1}
(Proposition 4.4, \cite{sourav1}) is useful for this paper.

\begin{prop}\label{lem:3}
Let $(S_1,S_2,P)$ be a $\Gamma_3$-contraction. Then for $i=1,2$,
$\Phi_i(\alpha S_1, \alpha^2 S_2, \alpha^3 P)\geq 0$ for all
$\alpha\in \overline{\mathbb D}$.
\end{prop}

Here is a set of characterizations for the $\Gamma_3$-unitaries
and for a proof of this result see Theorem 5.2 in \cite{sourav1}
or, Theorem 4.2 in \cite{BSR}.

\begin{thm}\label{thm:tu}
Let $(S_1, S_2, P)$ be a commuting triple of bounded operators.
Then the following are equivalent.

\begin{enumerate}

\item $(S_1,S_2,P)$ is a $\Gamma_3$-unitary,

\item $P$ is a unitary and $(S_1,S_2,P)$ is a
$\Gamma_3$-contraction,

\item $(\dfrac{2}{3}S_1,\dfrac{1}{3}S_2)$ is a
$\Gamma_2$-contraction, $P$ is a unitary and $S_1 = S_2^* P$.
\end{enumerate}
\end{thm}

\section{Proof of Theorem \ref{thm:decomp}}

First we consider the case when $P$ is a completely non-unitary
contraction. Then obviously $\mathcal H_1=\{0\}$ and if $P$ is a
unitary then $\mathcal H=\mathcal H_1$ and so $\mathcal
H_2=\{0\}$. In such cases the theorem is trivial. So let us
suppose that $P$ is neither a unitary nor a completely non unitary
contraction. With respect to the decomposition $\mathcal
H=\mathcal H_1\oplus \mathcal H_2$, let
\[
S_1=
\begin{bmatrix}
S_{111} & S_{112}\\
S_{121} & S_{122}
\end{bmatrix}\,,\,
S_2=
\begin{bmatrix}
S_{211} & S_{212}\\
S_{221} & S_{222}
\end{bmatrix}
\text{ and } P=
\begin{bmatrix}
P_1&0\\
0&P_2
\end{bmatrix}
\]
so that $P_1$ is a unitary and $P_2$ is completely non-unitary.
Since $P_2$ is completely non-unitary it follows that if
$h\in\mathcal H$ and
\[
\|P_2^nh\|=\|h\|=\|{P_2^*}^nh\|, \quad n=1,2,\hdots
\]
then $h=0$.

By the commutativity of $S_1$ and $P$ we obtain

\begin{align}
S_{111}P_1&=P_1S_{111}    & S_{112}P_2=P_1S_{112}\,, \label{eqn:1} \\
S_{121}P_1&=P_2S_{121}    & S_{122}P_2=P_2S_{122}\,. \label{eqn:2}
\end{align}

Also the commutativity of $S_2$ and $P$ gives

\begin{align}
S_{211}P_1&=P_1S_{211}    & S_{212}P_2=P_1S_{212}\,, \label{eqn:3} \\
S_{221}P_1&=P_2S_{221}    & S_{222}P_2=P_2S_{222}\,. \label{eqn:4}
\end{align}
By Proposition \ref{lem:3}, we have for all $\omega,
\beta\in\mathbb T$,
\begin{align*}
\Phi_1(\omega S_1,\omega^2 S_2,\omega^3
P)&=9(I-P^*P)+(S_1^*S_1-S_2^*S_2)-6\text{
Re }\omega(S_1-S_2^*P)\geq 0 \,,\\
\Phi_2(\beta S_1,\beta^2 S_2, \beta^3
P)&=9(I-P^*P)+(S_2^*S_2-S_1^*S_1)-6\text{ Re }\beta^2(S_2-S_1^*P)
\geq 0 \,.
\end{align*}
Adding $\Phi_1$ and $\Phi_2$ we get
\[
3(I-P^*P)-\text{Re }\omega(S_1-S_2^*P)-\text{Re }
\beta^2(S_2-S_1^*P)\geq 0
\]
that is
\begin{align}\label{eqn:5}
\begin{bmatrix}
0&0\\
0&3(I-P_2^*P_2)
\end{bmatrix}
-& \text{ Re }\omega
\begin{bmatrix}
S_{111}-S_{211}^*P_1 & S_{112}-S_{221}^*P_2\\
S_{121}-S_{212}^*P_1&S_{122}-S_{222}^*P_2
\end{bmatrix} \\
-&\text{ Re }\beta^2
\begin{bmatrix}
S_{211}-S_{111}^*P_1&S_{212}-S_{121}^*P_2\\
S_{221}-S_{112}^*P_1&S_{222}-S_{122}^*P_2
\end{bmatrix}\, \geq 0 \notag
\end{align}
for all $\omega,\beta\in\mathbb T$. Since the matrix in the left
hand side of (\ref{eqn:5}) is self-adjoint, if we write
(\ref{eqn:5}) as

\begin{equation}\label{eqn:6}
\begin{bmatrix}
R&X\\
X^*&Q
\end{bmatrix}
\geq 0\,,
\end{equation}
then

\begin{eqnarray*}\begin{cases}
&(\mbox{i})\; R\,, Q \geq 0 \text{ and } R=-\text{ Re }\omega (
S_{111}-S_{211}^*P_1) -\text{ Re }\beta^2
(S_{211}-S_{111}^*P_1)\\& (\mbox{ii}) X= -\frac{1}{2} \{ \omega (
S_{112}-S_{221}^*P_2)+\bar{\omega}(S_{121}^*-P_1^*S_{212})\\&
\quad \quad \quad \quad + \beta^2
(S_{212}-S_{121}^*P_2)+\bar{\beta^2}(S_{221}^*-P_1^*S_{112}) \}
\\&(\mbox{iii})\; Q=3(I-P_2^*P_2)-\text{ Re }\omega (S_{122}-S_{222}^*P_2) -
\text{ Re }\beta^2 (S_{222}-S_{122}^*P_2) \;.
\end{cases}
\end{eqnarray*}

Since the left hand side of (\ref{eqn:6}) is a positive
semi-definite matrix for every $\omega$ and $\beta$, if we choose
$\beta^2=1$ and $\beta^2=-1$ respectively then consideration of
the $(1,1)$ block reveals that
\[
\omega(S_{111}-S_{211}^*P_1)+\bar{\omega}(S_{111}^*-P_1^*S_{211})\leq
0
\]
for all $\omega\in\mathbb T$. Choosing $\omega =\pm 1$ we get
\begin{equation}\label{eqn:7}
(S_{111}-S_{211}^*P_1)+(S_{111}^*-P_1^*S_{211})=0
\end{equation}
and choosing $\omega =\pm i$ we get
\begin{equation}\label{eqn:8}
(S_{111}-S_{211}^*P_1)-(S_{111}^*-P_1^*S_{211})=0\,.
\end{equation}
Therefore, from (\ref{eqn:7}) and (\ref{eqn:8}) we get
\[
S_{111}=S_{211}^*P_1\,,
\]
where $P_1$ is unitary. Similarly, we can show that
\[
S_{211}=S_{111}^*P_1\,.
\]
Therefore, $R=0$. Since $(S_1,S_2,P)$ is a $\Gamma_3$-contraction,
$\|S_2\|\leq 3$ and hence $\|S_{211}\|\leq 3$. Also since
$(S_1,S_2,P)$ is a $\Gamma_3$-contraction, by Lemma 2.5 of
\cite{BSR} $(\dfrac{2}{3}S_1,\dfrac{1}{3}S_2)$ is a
$\Gamma_2$-contraction and hence
$(\dfrac{2}{3}S_{111},\dfrac{1}{3}S_{211})$ is a
$\Gamma_2$-contraction. Therefore, by part-(3) of Theorem
\ref{thm:tu}, $(S_{111},S_{211},P_1)$ is a $\Gamma_3$-unitary.\\

Now we apply Proposition 1.3.2 of \cite{bhatia} to the positive
semi-definite matrix in the left hand side of (\ref{eqn:6}). This
Proposition states that if $R,Q \geq 0$ then $\begin{bmatrix} R&X
\\ X^*&Q
\end{bmatrix} \geq 0$ if and only if $X=R^{1/2}KQ^{1/2}$ for
some contraction $K$.\\

\noindent Since $R=0$, we have $X=0$. Therefore,
\[
\omega (
S_{112}-S_{221}^*P_2)+\bar{\omega}(S_{121}^*-P_1^*S_{212})+
\beta^2
(S_{212}-S_{121}^*P_2)+\bar{\beta^2}(S_{221}^*-P_1^*S_{112}) =0\;,
\]
for all $\omega,\beta \in\mathbb T$. Choosing $\beta^2 =\pm 1$ we
get
\[
\omega (
S_{112}-S_{221}^*P_2)+\bar{\omega}(S_{121}^*-P_1^*S_{212})=0\;,
\]
for all $\omega \in \mathbb T$. With the choices $\omega=1,i$ ,
this gives
\[
S_{112}=S_{221}^*P_2\,.
\]
Therefore, we also have
\[
S_{121}^*=P_1^*S_{212}\,.
\]
Similarly, we can prove that
\[
S_{212}=S_{121}^*P_2\,,\quad S_{221}^*=P_1^*S_{112}\,.
\]
Thus, we have the following equations
\begin{align}
S_{112}&=S_{221}^*P_2         & S_{121}^*&=P_1^*S_{212} \label{eqn:9}\\
S_{212}&=S_{121}^*P_2         & S_{221}^*&=P_1^*S_{112}\,.
\label{eqn:10}
\end{align}
Thus from (\ref{eqn:9}), $S_{121}=S_{212}^*P_1$ and together with
the first equation in (\ref{eqn:2}), this implies that
\[
S_{212}^*P_1^2=S_{121}P_1=P_2S_{121}=P_2S_{212}^*P_1
\]
and hence
\begin{equation}\label{eqn:11}
S_{212}^*P_1=P_2S_{212}^*\,.
\end{equation}
From equations in (\ref{eqn:3}) and (\ref{eqn:11}) we have that
\[
S_{212}P_2=P_1S_{212}\,, \quad S_{212}{P_2^*}={P_1^*}S_{212}.
\]
Thus
\begin{align*}
S_{212}P_2{P_2^*} &=P_1S_{212}{P_2^*} =P_1{P_1^*}S_{212}
=S_{212}\,, \\
S_{212}{P_2^*}P_2 &= {P_1^*}S_{212}P_2
={P_1^*}P_1S_{212}=S_{212}\,,
\end{align*}
and so we have
\[
P_2{P_2^*}S_{212}^*=S_{212}^*={P_2^*}P_2S_{212}^*\,.
\]
This shows that $P_2$ is unitary on the range of $S_{212}^*$ which
can never happen because $P_2$ is completely non-unitary.
Therefore, we must have $S_{212}^*=0$ and so $S_{212}=0$.
Similarly we can prove that $S_{112}=0$. Also from (\ref{eqn:9}),
$S_{121}=0$ and from (\ref{eqn:10}), $S_{221}=0$. Thus with
respect to the decomposition $\mathcal H=\mathcal H_1\oplus
\mathcal H_2$
\[
S_1=
\begin{bmatrix}
S_{111}&0\\
0&S_{122}
\end{bmatrix}\,, \quad
S_2=
\begin{bmatrix}
S_{211}&0\\
0&S_{222}
\end{bmatrix}.
\]
So, $\mathcal H_1$ and $\mathcal H_2$ reduce $S_1$ and $S_2$. Also
$(S_{122},S_{222},P_2)$, being the restriction of the $\mathbb
E$-contraction $(S_1,S_2,P)$ to the reducing subspace $\mathcal
H_2$, is an $\Gamma_3$-contraction. Since $P_2$ is completely
non-unitary, $(S_{122},S_{222},P_2)$ is a completely non-unitary
$\Gamma_3$-contraction.

\end{document}